\documentclass[11pt,english]{amsart}

\usepackage[a4paper]{geometry}
\usepackage{mathrsfs}
\usepackage{amssymb}
\usepackage{amsmath}
\usepackage{dsfont}
\usepackage{nccmath}
\usepackage{fancyhdr}
\usepackage{color}
\usepackage{enumerate}
\usepackage{tikz}
\newtheorem{theorem}{Theorem}[section] 
\newtheorem{lemma}[theorem]{Lemma}     

\newtheorem{proposition}[theorem]{Proposition}

\newtheorem{definition}[theorem]{Definition}

\newtheorem{remark}[theorem]{Remark}

\newtheorem*{theorem*}{Theorem}

\usepackage[bookmarksopen, naturalnames]{hyperref}

\makeatother

\begin{document}

		\title[From frequent hypercyclicity to hypercyclicity]{Growth of entire functions: from frequent hypercyclicity to hypercyclicity for the differentiation operator}
	
	\author{A. Mouze}
	\thanks{}
	\address{Augustin Mouze, Univ. Lille, \'Ecole Centrale de Lille, CNRS, UMR 8524 - Laboratoire Paul Painlev\'e  F-59000 Lille, France}
	\email{augustin.mouze@univ-lille.fr}

	\keywords{frequently hypercyclic operator, entire functions, rate of growth}
	\subjclass[2020]{47A16, 30E25, 47B37, 47B38, 30D15}
	
	\begin{abstract} We study the rate of growth of entire functions that are frequently hypercyclic with respect to some upper weighted densities for the differentiation operator. The statements obtained show the link between the minimal growth of frequently hypercyclic entire functions and that of hypercyclic entire functions. The results are optimal. 
		
	\end{abstract}
	\maketitle                   
	
	\footnotetext{The author is partially supported by the grant ANR-24-CE40-0892-01 of the French
		National Research Agency ANR (project ComOp).}
	
	\section{Introduction} Throughout this article, we use the convention $\mathbb{N}:=\{1,2,3,\dots\}$. A continuous and linear operator $T$ on a separable Fr\'echet space $X$ is \textit{hypercyclic} if there exists $x\in X$ whose orbit $\{T^nx : n\in\mathbb{N}\}$ is dense in $X$. Such a vector $x$ is called a hypercyclic (or universal) element. This notion has been studied extensively and has led to many interesting results. We refer to the books \cite{BayMath,grossebook} and the references therein. Recently more restrictive notions have been introduced, those of frequent hypercyclicity \cite{BayGriv} and $\mathscr{U}$-frequent hypercyclicity \cite{Shka}, which measure in some way the frequency of visits of every non-empty open set by the orbits in terms of lower or upper density. The \textit{lower} and \textit{upper natural density} of a subset $A\subset \mathbb{N}$ are defined as 
	$$\underline{d}(A)=\liminf_{n\rightarrow \infty}\frac{\#(A\cap [1,n])}{n}\quad \hbox{and}\quad
	\overline{d}(A)=\limsup_{n\rightarrow \infty}\frac{\#(A\cap [1,n])}{n},$$ 
	respectively. 
	
	\begin{definition}{\rm Let $X$ be a separable Fr\'echet space and let $T:X\rightarrow X$ be an operator. A vector $x\in X$ is called \textit{frequently hypercyclic} (resp. $\mathscr{U}$\textit{-frequently hypercyclic}) for $T$ if, for every non-empty open subset $U$ of $X$ 
			$$\underline{d}(\{n\in \mathbb{N} : T^nx\in U\})>0\quad (\hbox{resp. } \overline{d}(\{n\in \mathbb{N} : T^nx\in U\})>0). $$
			The operator $T$ is called \textit{frequently hypercyclic} (resp. $\mathscr{U}$\textit{-frequently hypercyclic}) if it possesses a \textit{frequently hypercyclic} vector (resp. a $\mathscr{U}$\textit{-frequently hypercyclic} vector).} 
	\end{definition}
	
	Clearly a frequently hypercyclic operator is $\mathscr{U}$-frequently hypercyclic and a $\mathscr{U}$-frequently hypercyclic operator is hypercyclic. Classical example of frequently hypercyclic operator is the differentiation operator $D:H(\mathbb{C})\rightarrow H(\mathbb{C})$, given by $f\mapsto f'$ on the space $H(\mathbb{C})$ of entire functions on the complex plane $\mathbb{C}$, equipped with the standard compact-open topology \cite{McLane}. In this note, we are investigating the minimal growth rates, in terms of $L^p$-averages, of hypercyclic functions for $D$, considering the frequency with which their orbits visit non-empty open sets. Thus, for an entire function $f$ and $1\leq p\leq\infty$, let us consider, for $r>0$,
	$$M_p(f,r)=\left(\frac{1}{2\pi}\int_0^{2\pi}\vert f(re^{it})\vert ^p\right)^{1/p}\quad\hbox{ and }\quad M_{\infty}(f,r)=\sup_{\vert z\vert=r}\vert f(z)\vert.$$
	In the sequel, for $1\leq p\leq\infty$, let us denote by $q$ the conjugate exponent of $p$, i.e. $\frac{1}{p}+\frac{1}{q}=1$ (by convention $q=\infty$ if $p=1$ and $q=1$ if $p=\infty$). 
	The case of growth rate of hypercyclic functions for $D$ (or $D$-hypercyclic functions) was investigated in \cite{BBG,ge1,ge11}. The optimal result is given by the following statement.
	
	\begin{theorem}\label{DHC}\cite[Theorem 2.1]{BBG} Let $1\leq p\leq\infty$. For any function $\varphi:\mathbb{R}_+\rightarrow\mathbb{R}_+$ with $\varphi(r)\rightarrow\infty$ as $r\rightarrow\infty$, there is a $D$-hypercyclic entire function $f$ with $M_p(f,r)\leq \varphi (r)e^r r^{-1/2}$ for $r$ sufficiently large. On the other hand, every $D$-hypercyclic entire function $f$ satisfies $\limsup\limits_{r\rightarrow\infty}\left(e^{-r}\sqrt{r}M_p(f,r)\right)=\infty$. 
	\end{theorem}
	
	The case of growth rate of $D$-frequently hypercyclic functions was investigated in \cite{BBG,BB,Drasin}. The optimal result is given by the following statement. For $1\leq p\leq\infty$ let us denote 
	$$a(p):=\frac{1}{2\min(2,p)}=\max\left(\frac{1}{2p},\frac{1}{4}\right) =\frac{1}{2}-\frac{1}{2\max(2,q)}.$$
	Observe that for $p=1$  we have $a(1)=1/2$.
	
	\begin{theorem}\label{DFHC}\cite[Theorem 1.1]{Drasin} \begin{enumerate}
			\item Given $c>0$ and $p\in (1,\infty]$, there is an entire $D$-frequently hypercyclic function $f$ with $M_p(f,r)\leq ce^r r^{-a(p)}$ for all $r>0$. This estimate is optimal: every such function satisfies $\limsup\limits_{r\rightarrow\infty}\left(r^{a(p)}e^{-r}M_p(f,r)\right)>0$.
			\item  For any function $\varphi:\mathbb{R}_+\rightarrow\mathbb{R}_+$ with $\varphi(r)\rightarrow\infty$ as $r\rightarrow\infty$, there is an entire $D$-frequently hypercyclic function $f$ with $M_1(f,r)\leq \varphi(r)e^r r^{-1/2}$ for all $r>0$. This estimate is optimal: every such function satisfies $\limsup\limits_{r\rightarrow\infty}\left(\sqrt{r}e^{-r}M_1(f,r)\right)=\infty$.
		\end{enumerate}
	\end{theorem}
	
	Thus for $p>1$ these theorems allow us to make a distinction between $D$-hypercyclic functions and $D$-frequently hypercyclic functions. We can then interested in the growth rate of $\mathscr{U}$-frequently hypercyclic functions for $D$. Actually, it won't be difficult to justify that these functions and the $D$-frequently hypercyclic functions share the same optimal rate of growth. Therefore, in order to understand the difference of growth between $\mathscr{U}$-frequently hypercyclic and hypercyclic functions for $D$, we introduce intermediate notions of linear dynamics between $\mathscr{U}$-frequent hypercyclicity and hypercyclicity replacing the usual upper density by specific upper weighted densities. 
	
	\begin{definition} {\rm Let $\beta=(\beta_n)$ be a sequence of positive real numbers such that $\sum\limits_{k=1}^n\beta_k\rightarrow \infty$ as $n$ tends to $\infty$. For a subset $E\subset \mathbb{N}$, its upper $\beta$-density is given by
			$$\overline{d}_{\beta}(E)=\limsup_{n\rightarrow +\infty}\frac{\sum_{k=1}^n \beta_k\mathds{1}_E(k)}{\sum_{k=1}^n \beta_k},$$
			where $\mathds{1}_E(k)=1$ if $k\in E$ and $0$ otherwise.} 
	\end{definition}
	
	These quantities enjoy all the classical properties of asymptotic densities (see \cite{ErMo1,FrSa}) and allow to define dynamical notions of the same nature as the $\mathcal {U}$-frequent hypercyclicity. 
	
	\begin{definition} {\rm Let $\beta=(\beta_n)$ be a sequence of positive real numbers such that $\sum\limits_{k=1}^n\beta_k\rightarrow \infty$ as $n$ tends to $\infty$ and let $E$ be a subset of $\mathbb{N}$. Let $X$ be a separable Fr\'echet space and let $T:X\rightarrow X$ be an operator. The operator $T$ is said to be \textit{$\mathscr{U}_{\beta}$-frequently hypercyclic} if there exists $x\in X$ such that for every non-empty open subset $U\subset X$, $\overline{d}_{\beta}(\{n\in\mathbb{N}\ :\ T^nx\in U\})>0.$}
	\end{definition}
	
	Let us consider a density scale depending on a continuous parameter $\gamma$ given by the sequence of weights $\beta^\gamma=(e^{n^\gamma})$, with $0\leq \gamma\leq 1$, that connects the notions of $\mathscr{U}$-frequent hypercyclicity and hypercyclicity. Indeed the $\mathscr{U}_{\beta^{0}}$-frequent hypercyclicity coincides with the $\mathscr{U}$-frequent hypercyclicity and the $\mathscr{U}_{\beta^{1}}$-frequent hypercyclicity coincides with the hypercyclicity. Moreover, for any $0\leq \gamma\leq\gamma'\leq 1$ and for any subset $E\subset \mathbb{N}$, the following chain of inequalities $\overline{d}(E)=\overline{d}_{\beta^{0}}(E)\leq \overline{d}_{\beta^{\gamma}}(E)\leq\overline{d}_{\beta^{\gamma'}}(E)\leq\overline{d}_{\beta^{1}}(E)$ shows that the notions of $\mathscr{U}_{\beta^{\gamma}}$-frequent hypercyclicity for $\gamma\in(0,1)$ provide a refinement of the theory of linear dynamics between $\mathscr{U}$-frequent hypercyclicity and hypercyclicity. We refer to \cite[Section 4.1]{MouMunder}. \\
	
	For $p\in[1,\infty]$ and $0\leq \gamma\leq 1$, we set 
	$$\alpha(p,\gamma):=\frac{1}{2}-\frac{\min(1,2(1-\gamma))}{2\max(2,q)}.$$
	Observe that for $0\leq \gamma\leq \frac{1}{2}$, the exponent $\alpha(p,\gamma)$ is the exponent $a(p)$ of Theorem \ref{DFHC} and for $\gamma=1$, we have  $\alpha(p,1)=1/2$ which is the exponent of the non-exponential term of the rate of growth given in Theorem \ref{DHC}.\\
	
	The aim of the paper is to prove the following general result which links the growth of $\mathscr{U}$-frequently hypercyclic functions with that of hypercyclic functions for the differentiation operator. In view of the statement, this result thus connects Theorem \ref{DHC} with Theorem \ref{DFHC}.
	
	\begin{theorem}\label{maintheorem} Let $0\leq\gamma\leq 1$. 
		\begin{enumerate}
			\item Let $p\in (1,\infty]$.
			\begin{enumerate}
			\item If $0\leq\gamma<1$, given $c>0$ and $p\in (1,\infty]$, there is an entire $\mathscr{U}_{\beta^\gamma}$-frequently hypercyclic function $f$ for D with 
			$$M_p(f,r)\leq c\frac{e^r}{r^{\alpha(p,\gamma)}}.$$
			This estimate is optimal: every such function satisfies $\limsup\limits_{r\rightarrow\infty}\left(r^{\alpha(p,\gamma)}e^{-r}M_p(f,r)\right)>0$.
			\item If $\gamma=1$, given $\psi:\mathbb{R}_+\rightarrow \mathbb{R}_+$ be any increasing function with $\psi(r)\rightarrow \infty$ as $r\rightarrow\infty$, there is an entire ($\mathscr{U}_{\beta^1}$-frequently) hypercyclic function $f$ for $D$ with 
			$$M_p(f,r)\leq \psi(r)\frac{e^r}{r^{1/2}}.$$
			This estimate is optimal: every such function satisfies $\limsup\limits_{r\rightarrow\infty}\left(r^{1/2}e^{-r}M_p(f,r)\right)=\infty$.
		\end{enumerate}
			\item Let $p=1$. Given $\psi:\mathbb{R}_+\rightarrow \mathbb{R}_+$ be any increasing function with $\psi(r)\rightarrow \infty$ as $r\rightarrow\infty$, there is an entire $\mathscr{U}_{\beta^\gamma}$-frequently hypercyclic function $f$ for $D$ with 
			$$M_1(f,r)\leq \psi(r)\frac{e^r}{r^{1/2}}.$$
			This estimate is optimal: every such function satisfies $\limsup\limits_{r\rightarrow\infty}\left(r^{1/2}e^{-r}M_p(f,r)\right)=\infty$. 
		
		\end{enumerate}
	\end{theorem}
	
	Theorems \ref{DHC} and \ref{DFHC} cover the cases where $p=1$ or $\gamma=1$. For $p>1$ the preceding statement reveals that the optimal growth rate for  $\mathscr{U}_{\beta^{\gamma}}$-frequently hypercyclic functions for $D$ is continuously dependent on the parameter $\gamma$. To be more precise, for $0\leq \gamma\leq 1/2$ the $\mathscr{U}_{\beta^{\gamma}}$-frequently hypercyclic functions for $D$ can grow as slowly as the $\mathscr{U}$-frequently hypercyclic functions (and even as the frequently hypercyclic functions). Yet, starting from $\gamma=1/2$ this rate of growth depends on an affine function of the parameter, eventually reaching that of hypercyclic functions. Furthermore, in \cite{MouMunder}, the authors explored weighted Taylor shifts on the unit disk. In the same spirit they demonstrated, utilizing the same scale of densities, a continuous progression between the growth of $\mathscr{U}$-frequently hypercyclic functions and that of hypercyclic functions for theses Taylor shifts. However the key distinction lies in the fact that for these Taylor shifts on the unit disk, the modification in the rate of growth immediately occurred from the $\mathscr{U}$-frequent hypercyclicity, i.e. from $\gamma=0$, in contrast to our case of the differentiation operator where it begins at $\gamma=1/2$.\\

	The paper is organized as follows: in Section \ref{section2} we establish admissible growth in terms of $L^p$-averages of entire functions $\mathscr{U}_{\beta^{\gamma}}$-frequently hypercyclic functions for $D$ according to the value of the parameter $\gamma$. In Section \ref{section3} we construct as in \cite{Drasin} entire functions that show the optimality of the admissible growth obtained in the previous section. Finally we can prove Theorem \ref{maintheorem}.

	\section{Rate of growth of $\mathscr{U}$-frequently hypercyclic functions}\label{section2}
	Let $0\leq\gamma<1$. In this section, we study lower estimates on the possible growth rates for $\mathscr{U}_{\beta^\gamma}$-frequently hypercyclic functions for $D$. Notice that the case $\beta=0$ corresponds to the case of $\mathscr{U}$-frequent hypercyclicity. The method employed is inspired by that of the proof of \cite[Theorem 2.4]{BBG}. Actually the following statement holds.\\ 
	
	First of all, let us mention that throughout the paper we will use the following estimate obtained thanks to an integral comparison test:   
	\begin{equation}\label{estimafunda}
		\hbox{for }0<\gamma<1,\quad \displaystyle \sum_{k=1}^n e^{k^{\gamma}}\underset{n\rightarrow +\infty}{\sim} \frac{n^{1-\gamma}}{\gamma}e^{n^{\gamma}}.	
	\end{equation}
	
	\begin{theorem}\label{sens1}
		Let $1< p\leq\infty$, $q$ the conjugate exponent of $p$, $0\leq \gamma < 1$ and set $\alpha(p,\gamma)=\frac{1}{2}-\frac{\min(1,2(1-\gamma))}{2\max(2,q)}$. Let $\psi:\mathbb{R}_+\rightarrow \mathbb{R}_+$ be any decreasing function with $\psi(r)\rightarrow 0$ as $r\rightarrow\infty$. Then there is no $\mathscr{U}_{\beta^\gamma}$-frequently hypercyclic entire function $f$ for the differentiation operator that satisfies 
			\begin{equation}\label{estimlimsupgamma}
				M_p(f,r)\leq \psi(r)\frac{e^r}{r^{\alpha(p,\gamma)}}\ \ \hbox{ for }r>0\hbox{ sufficiently large.}
			\end{equation}
		\end{theorem}
		
		\begin{proof} First notice that a careful examination of the proof of \cite[Theorem 2.4]{BBG} allows us to claim that there is no $\mathscr{U}$-frequently hypercyclic entire function $f$ for the differentiation operator that satisfies $M_p(f,r)\leq \psi(r)\frac{e^r}{r^{a(p)}}$. Indeed it suffices to replace $\liminf$ by $\limsup$ at the end of the proof. Since $\alpha(p,0)=a(p)$, the theorem is proved  for $\gamma=0$.\\
			
			Now let us consider $0<\gamma<1$ and $p>1$. Since $M_2(f,r)\leq M_p(f,r)$ for $2\leq p\leq\infty$, it suffices to prove the result for $p\leq 2$. Thus assume $1<p\leq 2$. 
			
			\vskip1mm
			
			Let also $f=\sum_{n\geq 0}\frac{a_n}{n!}z^n$ be a $\mathscr{U}_{\beta^\gamma}$-frequently hypercyclic entire function for the differentiation operator. One can find an increasing sequence $(n_k)$ of positive integers with $\overline{d}_{\beta^\gamma}((n_k))>0$ such that $\vert D^{n_k}f(0)-3/2\vert=\vert a_{n_k}-3/2\vert<1/2$, which implies $\vert a_{n_k}\vert>1$. Set $I=\{n\in\mathbb{N};\ \vert a_n\vert>1\}$ and for all positive integer $m$, $I_m=I\cap [1,m]$. Clearly $\delta_\gamma:=\overline{d}_{\beta^\gamma}(I)>0$. 
			Let us consider an increasing sequence $(m_l)$ of positive integers such that 
			\begin{equation}\label{equaml}
				\frac{\gamma}{m_l^{1-\gamma}e^{m_l^\gamma}}\sum_{n\in I_{m_l}}e^{n^\gamma}\rightarrow \delta_{\gamma}.
			\end{equation}
			For every $0<\gamma<1$, let us also choose $\eta_1,$ $\eta_2$ and $\varepsilon$ that depend on $\gamma$ such that 
			\begin{equation}\label{choiceetaepsilon}
				0<\eta_1<1-e^{-\gamma},\quad 0<\eta_2<\frac{\delta_{\gamma}(1-e^{-\gamma}-\eta_1)}{e^{-\gamma}+\eta_1}\hbox{ and } 0<\varepsilon<\delta_{\gamma}(1-e^{-\gamma}-\eta_1)-\eta_2(e^{-\gamma}+\eta_1).
			\end{equation}
			\vskip2mm
			\begin{enumerate}[(a)]
				\item \textit{Case $0<\gamma\leq\frac{1}{2}$.}\\
				In this case, we have $\alpha(p,\gamma)=1/2-1/(2q)=1/(2p)$. Suppose that $f$ satisfies (\ref{estimlimsupgamma}). By the Hausdorff-Young inequality, we obtain
				$$\left( \sum_{n=0}^{+\infty }\vert a_n\vert ^q \frac{r^{qn}}{n!^q}\right)^{1/q}\leq M_p(f,r)\leq \psi(r)\frac{e^r}{r^{\frac{1}{2}-\frac{1}{2q}}},$$
				for $r>0$ sufficiently large. Thus we derive
				\begin{equation}\label{sumestim0}	
					\sum_{n=0}^{+\infty}\vert a_n\vert ^q \frac{r^{qn+q/2-1/2}e^{-qr}}{n!^q[\psi(r)]^q}\leq 1.		
				\end{equation}
				Thanks to Stirling formula, the function 
				$$r\mapsto g(r)=\frac{r^{qn+q/2-1/2}e^{-qr}}{n!^q}$$
				has its maximum at $a_n:=n+1/(2p)$ with $g(a_n)\sim \frac{(2\pi)^{-q/2}}{\sqrt{n}}$ and an inflection point at $b_n:=a_n+\left(n/q+1/(2pq)\right)^{1/2}$. On $[a_n,b_n]$, the function $g$ therefore dominates the linear function $h$ that satisfies $h(a_n)=g(a_n)$ and $h(b_n)=0$. We get $h:r\mapsto \frac{g(a_n)}{b_n-a_n} (-r+b_n)$.\\
				Now let $(m_l)$ be the sequence given by (\ref{equaml}). If $l$ is sufficiently large, for all  $m_l-\lfloor m_l^{1-\gamma}\rfloor\leq n\leq m_l$, we have $[a_n,b_n]\subset [m_l-\lfloor m_l^{1-\gamma}\rfloor,m_l+\lfloor m_l^{1-\gamma}\rfloor]$ since $1-\gamma\geq 1/2$. In the following, set $u_l:=m_l-\lfloor m_l^{1-\gamma}\rfloor$, $v_l:=m_l+\lfloor m_l^{1-\gamma}\rfloor$ and observe that $u_l\rightarrow\infty$ as $l$ tends to infinity. Hence, for these positive integers $n$, the following chain of inequalities holds:
				$$\int_{u_l}^{v_l}\frac{r^{qn+q/2-1/2}e^{-qr}}{n!^q[\psi(r)]^q} dr\geq \int_{a_n}^{b_n}\frac{h(r)}{[\psi(u_l)]^q}dr=\frac{(b_n-a_n)g(a_n)}{2[\psi(u_l)]^q}\geq  \frac{C_1}{[\psi(u_l)]^q},$$
				where $C_1>0$ is a constant. Now integrating (\ref{sumestim0}) over $[u_l,v_l]$, we derive, for $l$ large enough,
				\begin{equation}\label{estim3sums0}\frac{1}{e^{m_l^\gamma}}\sum_{n=u_l;\atop n\in I_{m_l}}^{m_l} e^{n^\gamma}\leq\frac{1}{e^{m_l^\gamma}}\sum_{n=u_l}^{m_l}\vert a_n\vert^q e^{n^\gamma}\leq\sum_{n=u_l}^{m_l}\vert a_n\vert^q\frac{e^{n^\gamma}}{e^{n^\gamma}}\leq \frac{[\psi(u_l)]^q}{C_1} \lfloor m_l^{1-\gamma}\rfloor.
				\end{equation}
				Let us write
				\begin{equation}\label{equsum0}
					\sum_{n=u_l;\atop n\in I_{m_l}}^{m_l} e^{n^\gamma}=\sum_{n\in I_{m_l}} e^{n^\gamma}-\sum_{ n\in I_{u_l-1}} e^{n^\gamma}.
				\end{equation}
				By hypothesis, for all $l$ large enough, we have
				$$\sum_{n\in I_{m_l}} e^{n^\gamma}\geq \left(\delta_{\gamma}-\frac{\varepsilon}{2}\right)\frac{m_l^{1-\gamma}e^{m_l^\gamma}}{\gamma}\hbox{ and }
				\sum_{ n\in I_{u_l-1}} e^{n^\gamma}\leq (\delta_{\gamma}+\eta_2)\frac{(u_l-1)^{1-\gamma}e^{(u_l-1)^\gamma}}{\gamma}.$$
				Thanks to the fact that 
				$$\left(\frac{u_l-1}{m_l}\right)^{1-\gamma}e^{(u_l-1)^\gamma-m_l^\gamma}\rightarrow e^{-\gamma},$$
				we have, for all $l$ sufficiently large,
				$$ \sum_{ n\in I_{u_l-1}} e^{n^\gamma}\leq \frac{m_l^{1-\gamma}e^{m_l^\gamma}}{\gamma}(\delta_{\gamma}+\eta_2)(e^{-\gamma}+\eta_1),$$
				and we derive, using (\ref{choiceetaepsilon}), for all $l$ sufficiently large,
				\begin{equation}\label{estim4sums0}\frac{\varepsilon}{2} \frac{m_l^{1-\gamma}e^{m_l^\gamma}}{\gamma}\leq \frac{m_l^{1-\gamma}e^{m_l^\gamma}}{\gamma}\left(\left(\delta_{\gamma}-\frac{\varepsilon}{2}\right)-(\delta_{\gamma}+\eta_2)(e^{-\gamma}+\eta_1)\right)
					\leq\left(\sum_{n\in I_{m_l}} e^{n^\gamma}-\sum_{ n\in I_{u_l-1}} e^{n^\gamma}\right).
				\end{equation}
				Thus combining (\ref{estim3sums0}), (\ref{equsum0}) with (\ref{estim4sums0}) we obtain, for all $l$ sufficiently large,
				$$ \frac{\varepsilon}{2} \frac{m_l^{1-\gamma}}{\gamma}\leq \frac{[\psi(u_l)]^q}{C_1} \lfloor m_l^{1-\gamma}\rfloor,$$
				which gives a contradiction since $\psi(u_l)\rightarrow 0$ as $l$ tends to infinity. Thus $f$ cannot satisfy (\ref{estimlimsupgamma}).
				
				\vskip4mm

				\item \textit{Case $\frac{1}{2}<\gamma<1$.}\\ 
				In this case, we have $\alpha(p,\gamma)=1/2-(1-\gamma)/q$. Suppose that $f$ satisfies (\ref{estimlimsupgamma}). By the Hausdorff-Young inequality, we obtain
				$$\left( \sum_{n=0}^{+\infty }\vert a_n\vert ^q \frac{r^{qn}}{n!^q}\right)^{1/q}\leq M_p(f,r)\leq \psi(r)\frac{e^r}{r^{\frac{1}{2}-\frac{1-\gamma}{q}}},$$
				for $r>0$ sufficiently large. Thus we derive
				\begin{equation}\label{sumestim2}	
					\sum_{n=0}^{+\infty}\vert a_n\vert ^q \frac{r^{qn+q/2-(1-\gamma)}e^{-qr}}{n!^q[\psi(r)]^q}\leq 1.		
				\end{equation}
				Thanks to Stirling formula, the function 
				$$r\mapsto g(r)=\frac{r^{qn+q/2-(1-\gamma)}e^{-qr}}{n!^q}$$
				has its maximum at $a_n':=n+1/2-(1-\gamma)/q$ with $g(a_n')\sim \frac{(2\pi)^{-q/2}}{n^{1-\gamma}}$ and an inflection point at $b_n':=a_n'+\left(n/q+1/(2q)-(1-\gamma)/q^2\right)^{1/2}$. We set $b_n''=n+n^{1-\gamma}-1$. Since $1/2<\gamma<1$, we have, for all $n$ sufficiently large, $a_n'\leq b_n''\leq b_n'$. Hence on $[a_n',b_n'']$, the function $g$ therefore dominates the linear function $h$ that satisfies $h(a_n')=g(a_n')$ and $h(b_n'')=0$. \\
				Now let $(m_l)$ be the sequence given by (\ref{equaml}). If $l$ is sufficiently large and $u_l=m_l-\lfloor m_l^{1-\gamma}\rfloor\leq n\leq m_l$, then $[a_n',b_n'']\subset [u_l,v_l]$, with $v_l=m_l+\lfloor m_l^{1-\gamma}\rfloor$. Hence, for such positive integers $n$ large enough we have, using the equality $b_n''-a_n'=n^{1-\gamma}-1/2+(1-\gamma)/q,$ 
				$$\int_{u_l}^{v_l}\frac{r^{qn+q/2-(1-\gamma)}e^{-qr}}{n!^q[\psi(r)]^q} dr\geq \int_{a_n'}^{b_n''}\frac{h(r)}{[\psi(u_l)]^q}dr=\frac{(b''_n-a_n')g(a_n')}{2[\psi(u_l)]^q}\geq \frac{C_2}{[\psi(u_l)]^q},$$
				where $C_2>0$ is a constant. Now integrating (\ref{sumestim2}) over $[u_l,v_l]$, we derive, for $l$ large enough,
				\begin{equation}\label{estim3sums1}\frac{1}{e^{m_l^\gamma}}\sum_{n=u_l;\atop n\in I_{m_l}}^{m_l} e^{n^\gamma}\leq\frac{1}{e^{m_l^\gamma}}\sum_{n=u_l}^{m_l}\vert a_n\vert^q e^{n^\gamma}\leq\sum_{n=u_l}^{m_l}\vert a_n\vert^q\frac{e^{n^\gamma}}{e^{n^\gamma}}\leq \frac{[\psi(u_l)]^q}{C_2} \lfloor m_l^{1-\gamma}\rfloor.
				\end{equation}
				The formula (\ref{estim3sums1}) is similar to (\ref{estim3sums0}). Therefore we can end the proof as in the previous case to obtain a contradiction. 
		
			\end{enumerate}

		\end{proof}

		\section{Optimality of the lower estimates and proof of the main theorem}\label{section3}
		This section is devoted to prove the main theorem (Theorem \ref{maintheorem}). First we are going to show the optimality of the growth rate given in Theorem \ref{sens1} by constructing for $\mathscr{U}_{\beta^\gamma}$-frequently hypercyclic functions for $1/2<\gamma<1$. This construction is based on several modifications of the one developed by Drasin and Saksman in \cite{Drasin} where the authors have explicitly exhibited $D$-frequently hypercyclic functions with the optimal growth rate. As in \cite{Drasin} we will need the so-called Rudin-Shapiro polynomials which have coefficients 
		$\pm 1$ (or bounded by $1$) and an optimal growth of $L^p$-norm. 
		
		\subsection{Preliminaries and notations}\label{notations_drasin}
		Let us recall the associated result in the form of\cite[Lemma 2.1]{Drasin} that summarized the result of Rudin-Shapiro \cite{RudSha}. For any polynomial $h$, let us define, for all $p\geq 1$,
		$$ \Vert h \Vert_p= \left(\int_0^{2\pi}\vert h(e^{it})\vert^p\right)^{1/p}dt\quad\hbox{ and }\quad \Vert h\Vert_{\infty}=\sup_{0\leq t\leq 2\pi}\vert h(e^{it})\vert.$$ 
		
		\begin{lemma}\label{lemma_rud_shap} 
			\begin{enumerate}
				\item For each $N\geq 1$, there is a trigonometric polynomial $p_N=\sum_{k=0}^{N-1}\varepsilon_{N,k}e^{ik\theta}$ 
				where $\varepsilon_{N,k}=\pm 1$ for all $0\leq k\leq N-1$ with at least half of the coefficients being $+1$ and with 
				$$\Vert p_{N}\Vert_{p}\leq 5\sqrt{N}\hbox{ for }p\in [2,+\infty].$$
				\item For each $N\geq 1$, there is a trigonometric polynomial $p_N^*=\sum_{k=0}^{N-1}a_{N,k}e^{ik\theta}$ 
				where $\vert a_{N,k}\vert\leq 1$ for all $0\leq k\leq N-1$ with at least $\lfloor \frac{N}{4} \rfloor$ coefficients being $+1$ and with 
				$$\|p^{*}_{N}\|_{p}\leq 3N^{1/q}\hbox{ for }p\in [1,2].$$
			\end{enumerate} 
		\end{lemma}
		
		We repeat some notations of \cite{Drasin}. 
		For any given polynomial $r$ with $r(z)=\sum_{j=0}^d \frac{r_j}{j!} z^j$ with $r_d\ne 0$ and  
		$d=\hbox{deg}(r)$, we set $\tilde{r}(z)=\sum_{j=0}^d r_jz^j$ and $\displaystyle \Vert \tilde{r}\Vert_{\ell^1}=\sum_{j=0}^d\vert r_j\vert$. 
		Denote by $\mathcal{P}=(q_k)$ the countable set of polynomials 
		with rational coefficients. Clearly $\mathcal{P}$ is a dense set in $H(\mathbb{C})$. Let us also consider 
		pairs $(q_k,l_k)$ with $d_k:=\hbox{deg}(q_k)$, $l_1=\displaystyle \Vert \tilde{q_1}\Vert_{\ell^1}$ and for $k\geq 2$, $l_k=\max(\displaystyle \Vert \tilde{q}_k\Vert_{\ell^1},1+l_{k-1})$. Thus $(l_k)$ is an increasing sequence with $l_k\rightarrow +\infty$ as $k\rightarrow +\infty$ and 
		$$\forall k\geq 1, \quad\Vert \tilde{q}_k\Vert_{\ell^1}\leq l_k.$$ 
		
		We set $\displaystyle 2\mathbb{N}=\bigcup_{k\geq 1}\mathcal{A}_{k}$ where for any $k\geq 1,$ 
		$\displaystyle \mathcal{A}_{k}=\left\{2^{k}(2j-1);j\in\mathbb{N} \right\}.$

		\subsection{Construction of a specific function}
		Given $0<c<1$ and $\frac{1}{2}<\gamma<1$. We set $\displaystyle f=\sum\limits_{n\geq 0}P_{n}$ where the sequence of blocks $(P_{n})$ are polynomials. The key point is the minimal growth and the precise size of the blocks $P_n$ to obtain both the desired growth rate and the $U_{\beta^\gamma}$-frequent hypercyclicity of the sum $f$. In order to have the slowest possible growth in terms of $L^p$-averages we use Rudin-Shapiro polynomials given by Lemma \ref{lemma_rud_shap}, whose degree depends on the values of $p$ and $\gamma$. Moreover a good choice of the degrees will ensure the $U_{\beta^\gamma}$-frequent hypercyclicity of the sum $f$. Hence we define the polynomials $P_n$ as follows
		\begin{enumerate}
			\item for $2\leq p\leq\infty,$ we set
			\begin{equation}\label{poly_2inf}
				\tilde{P}_{n}(z)  =  \left\{\begin{array}{l} 
					0 \hbox{ if } n \mbox{ is odd }\\
					0 \hbox{ if } n\in \mathcal{A}_{k}\hbox{ and } n<10 \alpha_{k}\\
					z^{n^2}p_{ \lfloor \frac{n^{2(1-\gamma)}}{\alpha_k}\rfloor } (z^{\alpha_k})\tilde{q}_k(z) \hbox{ otherwise }
			\end{array}\right.
		\end{equation}
		with for any $k\geq 1$, $\alpha_{k}=1+\left\lfloor \max\left(2d_k+8l_k,(10^5l_k/c)^2   \right)\right\rfloor$;
		\item for $1< p<2,$ we set
		\begin{equation}\label{poly_2inf}
			\tilde{P}_{n}(z)  =  \left\{\begin{array}{l} 
				0 \hbox{ if } n \mbox{ is odd }\\
				0 \hbox{ if } n\in \mathcal{A}_{k}\hbox{ and } n<10 \alpha_{k}\\
				z^{n^2}p_{ \lfloor \frac{n^{2(1-\gamma)}}{\alpha_k}\rfloor }^* (z^{\alpha_k})\tilde{q}_k(z) \hbox{ otherwise }
			\end{array}\right.
		\end{equation}
		with for any $k\geq 1$, $\alpha_{k}=1+\left\lfloor \max\left(2d_k+8l_k,(10^5l_k/c)^q   \right)\right\rfloor$.
	\end{enumerate}
	We denote by $q_k=\sum_{m=0}^{d_k}\frac{q_{k,m}}{m!}z^m$ and $p_{ \lfloor \frac{n^{2(1-\gamma)}}{\alpha_k}\rfloor }$ (or $p_{ \lfloor \frac{n^{2(1-\gamma)}}{\alpha_k}\rfloor }^*)=\sum_{m=0}^{\lfloor \frac{n^{2(1-\gamma)}}{\alpha_k}\rfloor}p_{k,\gamma,m}z^m$. 
	Under these hypotheses, since for all $i=0,\dots, \lfloor \frac{n^{2(1-\gamma)}}{\alpha_k}\rfloor-1$, $d_k+n^2+i\alpha_k<n^2+(i+1)\alpha_k$, observe that, for $ n\in \mathcal{A}_{k}$ with $n \geq \alpha_{k}$, we can write
	$$P_n=\sum_{i=0}^{\lfloor\frac{n^{2(1-\gamma)}}{\alpha_k}\rfloor}p_{k,\gamma,i}\left(\sum_{j=0}^{d_k}\frac{q_{k,j}}{(j+n^2+i\alpha_k)!}z^{j+n^2+i\alpha_k}\right).$$ Moreover, for $s=n^2+l\alpha_k$, with $0\leq l\leq \lfloor\frac{n^{2(1-\gamma)}}{\alpha_k}\rfloor$, we have
	\begin{equation}\label{equa_deriv_freq}
	\left(\frac{d}{dz}\right)^s(P_n)=p_{k,\gamma,l}q_k(z)+
	\sum_{i=l+1}^{\lfloor\frac{n^{2(1-\gamma)}}{\alpha_k}\rfloor}p_{k,\gamma,i}\left(\sum_{j=0}^{d_k}\frac{q_{k,j}}{(j+(i-l)\alpha_k)!}z^{j+(i-l)\alpha_k}\right).
	\end{equation}	
At last by construction the supports of $P_n$ are disjoint.
	
	\vskip2mm
	
	\subsection{Proof of the $\mathscr{U}_{\beta^\gamma}$-frequent hypercyclicity of the constructed function}
	The next proposition gives an estimate on the $L^p$-averages of $f$. 
	
	\begin{proposition}\label{sens21} Let $\frac{1}{2}<\gamma<1$. For $1<p\leq\infty$, the following estimate holds:
		$$M_p(f,r)\leq c\frac{e^r}{r^{\frac{1}{2}-\frac{1-\gamma}{\max(2,q)}}}.$$	
	\end{proposition}
	
	\begin{proof}
		Let first $2\leq p\leq\infty$ (so $\max(2,q)=2$). It is enough to consider only the case $p=\infty$. First since $0<2(1-\gamma)<1$, observe that Proposition 3.4 of \cite{Drasin} remains true in our context with the blocks $P_n$ defined above. Thus using Proposition 3.4 of \cite{Drasin}, Lemma \ref{lemma_rud_shap} and the hypotheses, we get, for $n\in\mathcal{A}_k$ positive and even,
		$$\begin{array}{rcl}M_{\infty}(P_n,n^2)\leq 10 e^{n^2}n^{-1}\Vert p_{ \lfloor \frac{n^{2(1-\gamma)}}{\alpha_k}\rfloor } (z^{\alpha_k})\tilde{q}_k(z)\Vert_{\infty}&
			\leq &  10e^{n^2}n^{-1}5\sqrt{ \frac{n^{2(1-\gamma)}}{\alpha_k}}\Vert \tilde{q}_k\Vert_{\ell^1}\\&\leq& 10 e^{n^2}n^{-\gamma}5 l_k\alpha_k^{-1/2}\\&\leq&10^{-4} e^{n^2}n^{-\gamma} 5 c\\&\leq&c 10^{-3} e^{n^2}n^{-\gamma},\end{array}$$
		and, in the similar manner,
		$$M_{\infty}(P_n,(n+1)^2)\leq c 10^{-3}e^{(n+1)^2}(n+1)^{-\gamma}.$$
		Next notice that $P_n=0$ for odd $n$. Thus applying Proposition 3.3 of \cite{Drasin} with $a=\gamma/2\in (0,1/2)$ we get
		$$M_{\infty}(f,r)\leq c e^r r^{-\gamma/2}=c e^r r^{-(1/2-(1-\gamma)/2)}.$$
		Finally let $1<p<2$ (so $\max(2,q)=q$). Similar estimates as the previous case lead to the following inequalities: for $n\in\mathcal{A}_k$ positive and even (for odd $n$ $P_n=0$), 
		$$\begin{array}{rcl}M_{p}(P_n,n^2)\leq 10 e^{n^2}n^{-1}\Vert p_{ \lfloor \frac{n^{2(1-\gamma)}}{\alpha_k}\rfloor }^* (z^{\alpha_k})\Vert_{p}\Vert \tilde{q}_k(z)\Vert_{\ell^1}&
			\leq& 10 e^{n^2}n^{-1}3\left(\frac{n^{2(1-\gamma)}}{\alpha_k}\right)^{1/q}l_k\\&\leq& 10^{-4} 3c e^{n^2}n^{-(1-\frac{2(1-\gamma)}{q})} \\&\leq& 10^{-3} c e^{n^2}n^{-(1-\frac{2(1-\gamma)}{q})},\end{array}$$
		and, in the same manner,
		$$M_{\infty}(P_n,(n+1)^2)\leq 10^{-3} c e^{(n+1)^2}(n+1)^{-(1-\frac{2(1-\gamma)}{q})}.$$
		Thus applying Proposition 3.3 of \cite{Drasin} with $a=\frac{1}{2}-\frac{(1-\gamma)}{q}\in (0,1/4)$ we get
		$$M_{\infty}(f,r)\leq ce^r r^{-(\frac{1}{2}-\frac{(1-\gamma)}{q})}.$$	
	\end{proof}
	
	Now we prove the $\mathscr{U}_{\beta^\gamma}$-frequent hypercyclicity of $f$. 
	
	\begin{proposition}\label{sens22} Let $\frac{1}{2}<\gamma<1$. The function $f$ is $\mathscr{U}_{\beta^\gamma}$-frequently hypercyclic.
	\end{proposition}
	
	\begin{proof} To prove the hypercyclicity of $f$, we argue as in \cite{Drasin} (p. 3684-3685). We give the ideas of the proof for $2\leq p\leq \infty$ (for $1<p<2$ the proof is similar). We set $f=\sum_{j\geq 0}\frac{a_j}{j!}z^j$. For any fixed even integer $n\in\mathcal{A}_k$ with $n\geq \alpha_k$, denote by $\mathcal{B}_n$ the set of indices $n$ such that the coefficient of $z^s$ in the polynomial $z^{n^2}p_{\lfloor \frac{n^{2(1-\gamma)}}{\alpha_k}\rfloor}(z^{\alpha_k})$ is $1$. Using the equality (\ref{equa_deriv_freq}) and the definition of $\alpha_k$ ($\alpha_k\geq 8l_k$), we get for all $k\geq 1$, with $s=n^2+l\alpha_k$,
		$$\begin{array}{rcl}\displaystyle\left(\frac{d}{dz}\right)^sf(z)-q_k(z)&=&\displaystyle
		\sum_{i=l+1}^{\lfloor\frac{n^{2(1-\gamma)}}{\alpha_k}\rfloor}p_{k,\gamma,i}\left(\sum_{j=0}^{d_k}\frac{q_{k,j}}{(j+(i-l)\alpha_k)!}z^{j+(i-l)\alpha_k}\right)+\sum_{m\geq n+2}P_m\\&=&\displaystyle
		\sum_{i=s+8l_k}^{(n+1)^2-1}\frac{a_i}{(i-s)!}z^{i-s}+\sum_{i=(n+2)^2}^{+\infty}\frac{a_i}{(i-s)!}z^{i-s}
		 .\end{array}$$
	We end as in \cite{Drasin} to obtain	
		$$\sup_{\vert z\vert=l_k}\left\vert q_k(z)-\left(\frac{d}{dz}\right)^sf(z)\right\vert\leq\frac{1}{l_k}\hbox{ for any }s\in\mathcal{B}_n,\ n\in\mathcal{A}_k.$$
		Thus since $(q_k)$ is a dense family in the space of entire functions, the function $f$ is hypercyclic for the differentiation operator.\\
		Now to conclude it remains to prove that for any fixed $k\geq 1$ the set 
		$$\mathcal{T}_k=\{s:s\in\mathcal{B}_n,\ n\in\mathcal{A}_k,\ n\geq \alpha_k\}$$ 
		has positive upper $\beta^{\gamma}$-density. To do this, first observe that for any fixed even integer $n\in\mathcal{A}_k$ $\max(\mathcal{B}_n)\leq n^2+\alpha_k \lfloor \frac{n^{2(1-\gamma)}}{\alpha_k}\rfloor\leq n^2+ \lfloor n^{2(1-\gamma)}\rfloor$ and, thanks to Lemma \ref{lemma_rud_shap} at least half of the coefficients of $p_{\lfloor \frac{n^{2(1-\gamma)}}{\alpha_k}\rfloor}$ is equal to $1$. 

		Hence we get 
		\begin{equation}\label{dgammaupp1}\frac{1}{\sum\limits_{j\leq \max(\mathcal{B}_n)}e^{j^\gamma}}\sum\limits_{j\in \mathcal{B}_n;\atop j\leq \max(\mathcal{B}_n) }e^{j^\gamma}\geq \frac{1}{\sum\limits_{j\leq n^2+\lfloor n^{2(1-\gamma)}\rfloor }e^{j^\gamma}}\sum\limits_{j=n^2}^{n^2+\lfloor 2^{-1}\lfloor\frac{n^{2(1-\gamma)}}{\alpha_k}-1\rfloor\rfloor}e^{j^\gamma}.
		\end{equation}
		We write 
		\begin{equation}\label{dgammaupp2}\sum\limits_{j=n^2}^{n^2+\lfloor 2^{-1}\lfloor\frac{n^{2(1-\gamma)}}{\alpha_k}-1\rfloor\rfloor}e^{j^\gamma}=\sum\limits_{j=1}^{n^2+\lfloor2^{-1}\lfloor\frac{n^{2(1-\gamma)}}{\alpha_k}-1\rfloor\rfloor}e^{j^\gamma}-\sum\limits_{j=1}^{n^2-1}e^{j^\gamma}.
		\end{equation}

		On one hand, thanks to the estimate (\ref{estimafunda}), it is easy to check that
		\begin{equation}\label{dgammaupp3}\frac{\sum\limits_{j\leq n^2+\lfloor 2^{-1}\lfloor\frac{n^{2(1-\gamma)}}{\alpha_k}-1\rfloor\rfloor}e^{j^\gamma}}{\sum\limits_{j\leq n^2+\lfloor n^{2(1-\gamma)}\rfloor }e^{j^\gamma}}\rightarrow e^{-\gamma(\frac{1}{2\alpha_k}-1)},\hbox{ as }n\longrightarrow +\infty.
		\end{equation}	
		On the other hand, thanks to the estimate (\ref{estimafunda}), it is easy to check again that
		\begin{equation}\label{dgammaupp4}\frac{\sum\limits_{j\leq n^2-1}e^{j^\gamma}}{\sum\limits_{j\leq n^2+\lfloor n^{2(1-\gamma)}\rfloor }e^{j^\gamma}}\rightarrow e^{-\gamma},\hbox{ as }n\longrightarrow +\infty.
		\end{equation}
		Combining (\ref{dgammaupp1}) with (\ref{dgammaupp2}), (\ref{dgammaupp3}) and (\ref{dgammaupp4}), we deduce
		$$\limsup_{n\rightarrow\infty}\left(\frac{1}{\sum\limits_{j\leq \max(\mathcal{B}_n)}e^{j^\gamma}}\sum\limits_{j\in \mathcal{B}_n;\atop j\leq \max(\mathcal{B}_n) }e^{j^\gamma}\right) \geq e^{-\gamma}\left(e^{\frac{1}{2\alpha_k}}-1\right)$$
		which implies $\overline{d}_{\beta^\gamma}(\mathcal{T}_k)>0$. This finishes the proof.
	\end{proof}
	\vskip3mm
	
	\noindent Finally we are ready to prove Theorem \ref{maintheorem}. 	 \\

	\noindent\textit{Proof of Theorem \ref{maintheorem}}\\
		Let $p>1$. 
		\begin{enumerate}[(a)]
			\item For $0\leq\gamma\leq 1/2$, we have $\alpha(p,\gamma)=\frac{1}{2}-\frac{1}{2\max(2,q)}$ which gives $\alpha(p,\gamma)=1/4$ if $p\geq 2$ and $1/(2p)$ if $1<p<2$. Thus the result is a direct consequence of Theorem \ref{sens1} and Theorem \ref{DFHC}.	
			\item For $1/2<\gamma < 1$, the result is a direct consequence of Theorem \ref{sens1} and the combination of Proposition \ref{sens21} with Proposition \ref{sens22}.
			\item For $\gamma=1$, the results is nothing more than Theorem \ref{DHC} already known.	
		\end{enumerate}
		For $p=1$, observe that, for all $0<\gamma\leq 1$, we have $\alpha(1,\gamma)=1/2$, hence the result immediately follows of Theorem \ref{DHC} and Theorem \ref{DFHC}.\\
		 \null\hfill$\square$
\vskip5mm

\begin{remark} {\rm Let $\delta$ be a lower or upper asymptotic density in the sense given by \cite{FrSa}. An operator $T:X\rightarrow X$, where $X$ is a separable Fr\'echet space, is said to be $\delta$-frequently hypercyclic if there exists $x\in X$ such that for every non-empty open subset $U\subset X$, $\delta(\{n\in\mathbb{N}\ :\ T^nx\in U\})>0.$ If we assume that, for every subset $E\subset\mathbb{N}$, the following inequality holds 
		$$\underline{d}(E)\leq\delta(E)\leq\overline{d}_{\beta^{1/2}}(E),$$
	then a frequently hypercyclic vector is $\delta$-frequently hypercyclic and a $\delta$-frequently hypercyclic vector is $\overline{d}_{\beta^{1/2}}$-frequently hypercyclic. Therefore, under the previous hypothesis, Theorems \ref{DFHC} and \ref{maintheorem} ensure that the $\delta$-frequently hypercyclic entire functions for the differentiation operator share the same optimal rate of growth as frequently hypercyclic entire functions. Examples of such densities can be found in \cite{ErMo1}.}
\end{remark}

\end{document}